\documentclass{article}

\usepackage{microtype}
\usepackage{graphicx}
\usepackage{subcaption}
\usepackage{booktabs} 

\usepackage{hyperref}

\newcommand{\R}{\mathbb{R}}
\newcommand{\N}{\mathbb{N}}

\newcommand{\E}{\mathbb{E}}
\newcommand{\PP}{\mathbb{P}}
\newcommand{\Var}{\mathrm{Var}}
\newcommand{\Cov}{\mathrm{Cov}}

\newcommand{\HH}{\mathcal{H}}

\newcommand{\Dis}{\mathbb{D}}

\newcommand{\LS}{\mathcal{L}}
\newcommand{\X}{\mathcal{X}}

\newcommand{\op}{\mathcal{T}}
\newcommand{\nn}{n}

 \usepackage[preprint]{icml2026}

\usepackage{amsmath}
\usepackage{amssymb}
\usepackage{mathtools}
\usepackage{amsthm}

\usepackage[capitalize,noabbrev]{cleveref}

\theoremstyle{plain}
\newtheorem{theorem}{Theorem}[section]

\newtheorem{lemma}[theorem]{Lemma}

\theoremstyle{definition}
\newtheorem{definition}[theorem]{Definition}

\theoremstyle{remark}
\newtheorem{remark}[theorem]{Remark}

\usepackage[textsize=tiny]{todonotes}

\icmltitlerunning{The general regularisation scheme applied to conditional density estimation}

\begin{document}

\twocolumn[
  \icmltitle{The general regularisation scheme applied to conditional density estimation}



  \icmlsetsymbol{equal}{*}

  \begin{icmlauthorlist}
    \icmlauthor{Gilles Germain}{yyy}
  \end{icmlauthorlist}

  \icmlaffiliation{yyy}{Department of Statistics, University of Oxford, Oxford, United Kingdom}

  \icmlcorrespondingauthor{Gilles Germain}{gilles.germain@stats.ox.ac.uk}

  \icmlkeywords{General regularisation scheme, nonparametric estimation, conditional density estimation, kernel methods}

  \vskip 0.3in
]



\printAffiliationsAndNotice{}  

\begin{abstract}
The general regularisation scheme, a versatile approach for nonparametric estimation, has been successfully applied to regression, density ratio, and score estimation. In this paper, we introduce a unified framework encompassing these settings and extend it to conditional density estimation, deriving a new estimator with rigorously established convergence rates. We implement the Landweber regularisation, which is computationally more tractable than Tikhonov regularisation in this context. 
Numerical experiments demonstrate that our estimator matches or outperforms the Nadaraya-Watson estimator in various scenarios, including time series models.
\end{abstract}
\section{Introduction}

This paper investigates the use of kernel-based regularisation methods for conditional density estimation. 
Regularisation aims at addressing ill-posed inverse problems (see e.g.\ \cite{Engl}). Starting from \cite{Evgeniou}, it has proved to be a powerful tool in nonparametric regression and learning theory. Among the various choices of hypothesis spaces, reproducing kernel Hilbert spaces (RKHS) have become particularly prominent. 
The earliest regularisation methods studied in this context were the Tikhonov method, also known as least-squares regularisation \cite{DeVito, DeVito06}, and the Landweber iteration, which relies on early stopping gradient descent \cite{Yao}. 
The general regularisation scheme (GRS), originally developed in \cite{Bakus} and introduced in this context by \cite{Verri, Bauer}, provides an abstract framework that unifies a broad class of regularisation methods. It encompasses the Tikhonov and Landweber approaches, as well as spectral cut-off techniques, 
the $\nu$-method, and iterated Tikhonov regularisation (see e.g. \cite{Verri}).  
Other approaches, such as Ivanov and Morozov regularisation, do not seem compatible with GRS  \cite{Page, Oneto}. 
Several strategies have been proposed for the data-driven selection of the regularisation parameter, including the balancing principle \cite{DeVito10, Lu} and the quasi-optimality criterion \cite{Kinderman}. Optimal convergence rates have been established for Tikhonov regularisation \cite{Caponnetto07} and for GRS \cite{Lu, Rastogi}. In \cite{Blanchard}, these ideas were extended to inverse learning problems, where optimal lower bounds on convergence rates were derived. Further extensions include linear functional regression \cite{Gupta, Lin}, polynomial functional regression \cite{Holzleitner}, and distributed learning \cite{Guo}. More recently, a growing body of work has compared implicit and explicit regularisation in linear regression (see, e.g., \cite{Ali, Wu}).  
Beyond regression, GRS has been applied to other statistical tasks. Density ratio estimation was first addressed in \cite{Que, Kanamori}, with pointwise convergence guarantees established in \cite{Nguyen}, while quantile regression was studied in \cite{Li07} and score estimation in \cite{Sriperumbudur, Zhou}. 

In this work, we focus on applying the GRS framework to conditional density estimation (CDE). While regression aims at estimating the conditional mean of a response variable $Y$ given covariates $X$, many applications---including risk and asset management \cite{Rothfuss} and renewable energy forecasting \cite{Shi21}---require a more complete description of the uncertainty associated with the mean response. One way to achieve this is by inferring the conditional distribution of $Y$ given $X$, which is the subject of CDE.  
A wide range of approaches has been proposed, such as local linear estimators \cite{Fan}, finite mixture models \cite{Figueiredo}, tree-based methods \cite{Gao}, neural networks \cite{Kostic}, and mixture density networks \cite{Graffeuille}. In this paper, we restrict attention to kernel-based methods. In a parametric setting, \cite{Fu11} employed kernel principal component analysis, while \cite{Alquier} proposed an estimator defined as the minimiser of the maximum mean discrepancy with respect to the empirical distribution of $(X,Y)$. In the nonparametric setting, the Nadaraya--Watson estimator relies on Bayes' rule and kernel density estimators of the joint and marginal distributions of $(X,Y)$ (see e.g.\ \cite{Hyndman}), whereas \cite{Schuster, Spiteri} reconstruct conditional densities from kernel mean embeddings.  

The paper is organised as follows. In Section~\ref{sec GRS}, we introduce an abstract statistical framework in which the GRS can be applied, and we review previously studied settings, including regression, density ratio estimation, and score estimation. We recall the main regularisation methods within this framework and establish consistency of the resulting estimator under simplified assumptions. For the Landweber iteration, we further propose selecting the learning rate via exact line search, which accelerates the convergence of the gradient descent. 
In Section~\ref{sec CDE}, we apply GRS to CDE 
and compare the resulting estimator with existing kernel-based CDE on several numerical examples, including Gaussian mixture and time series models.  

To summarise the novelties of this paper, we (i) propose a general statistical framework suitable for the application of GRS that unifies several existing settings, 
(ii) apply GRS to CDE and derive a novel estimator with well-established convergence rates. 

\section{The general regularisation scheme} 
\label{sec GRS}
\subsection{A general framework for GRS}
Let $V$ be a random vector with values in $\mathcal{V}$ and $f_*:\mathcal{Z}\to \R$ be a function related to $V$, where $\mathcal{V}$ and $\mathcal{Z}$ are closed Euclidean sets. For example, $f_*$ could be the density of $V$ or  the conditional mean of one component of 
$V$ given the others. 
Let $Z$ be another random vector with values in $\mathcal{Z}$. 
Our goal is to estimate $f_*$ nonparametrically from a sample $(v_i)_{i=1}^{n}$ of i.i.d.\ observations of $V$ and a sample $(z_i)_{i=1}^{n_z}$ of observations of $Z$. 
We will consider two alternative settings: 
\begin{enumerate}
\item[(S1)] $(z_i)_{i=1}^{n_z}$ is a set of i.i.d.\ observations of $Z$,
\item[(S2)] $Z=(X,U)$ and $(x_i)_{i=1}^{n}$ and $(u_j)_{j=1}^{n_u}$ are i.i.d.\  observations of $X$ and $U$ respectively. 
We set 
\begin{equation}
\label{eq zi}
z_i=\left(x_{\lceil \frac{i}{n_u}\rceil},u_{i-n_u\lfloor \frac{i-1}{n_u}\rfloor}\right)
\end{equation}
for $i=1,\ldots,n_z$ with $n_z=n_un$.
\end{enumerate}
Let $\HH$ be a reproducing kernel Hilbert space with kernel $k:\mathcal{Z}\times\mathcal{Z}\to \R$. We denote by $L^2(Z)$ the space of square-integrable functions with respect to $Z$ and by $\left\langle f,g\right\rangle_Z=\E[f(Z)g(Z)]$ and $\left\langle f,g\right\rangle_\HH$ the scalar products in $L^2(Z)$ and $\HH$ respectively. 
At this stage, we make the following assumptions: 
\begin{enumerate}
\item[(A1)]$f_*\in L^2(Z)$, 
\item[(A2)] $\HH\subset L^2(Z)$, $\HH$ is separable, $k$ is symmetric and $k\leq \kappa^2$ for some $\kappa>0$,
\item[(A3)] There exists a linear operator $\op:L^2(Z)\to L^1(V)$ such that $\E\left[f(Z)f_*(Z)\right]=\E\left[\op f(V)\right]$ for all $f\in L^2(Z)$ and $\op$ admits an explicit expression, 
\item[(A4)] $\left\|\op_2 k(\cdot,v)\right\|_\HH\leq C$ for all $v\in \mathcal{V}$ and some $C>0$ where $\op_2 k(z,v):=\op[z'\mapsto k(z,z')](v)$. 
\end{enumerate}
We set $\Dis(f,f_*):=\left\|f-f_*\right\|_{Z}$. 
Minimising  $f\mapsto\Dis(f,f_*)$ over $L^2(Z)$ is infeasible because its Fréchet gradient in $L^2(Z)$ is $2(f-f_*)$, which is not available. 
Instead, we seek to estimate the best approximation of $f_*$ within $\HH$ that is, the minimal norm solution of
\begin{equation}
\label{eq fH}
\inf_{f\in\HH}\left\|f-f_*\right\|_{Z},
\end{equation}
which we denote by $f_\HH$. 
A sufficient condition for existence and uniqueness of $f_\HH$ 
is that $Pf_*\in \HH$ where $P$ is the projection on the closure of $\HH$ 
(see e.g.\ \cite{Bauer}). In the terminology of inverse problems theory, $f_\HH$ is the Moore-Penrose solution of $I_k f=f_*$ where $I_k:\HH\to L^2(Z)$ is the inclusion operator from $\HH$ into $L^2(Z)$. 
We recall that the kernel operator $\LS:L^2(Z)\to \HH$ associated to $k$ and $Z$ is defined by
\begin{equation}
\label{eq LS}
\LS f(z)=\E\left[k(z,Z)f(Z)\right].
\end{equation}
Since the Fréchet gradient of $f\mapsto \Dis(f,f_*)$ in $\HH$ is $\LS f-\LS f_*$, the first order optimality condition is $\LS f_\HH=\LS f_*$. 
Moreover, we have by Assumption (A3) that
\begin{equation}
\label{eq b}
\LS f_*(z)=\E[\op_2 k(z,V)]=:b(z). 
\end{equation}
We thus have $\LS f_\HH=b$ where $b$ can be estimated using the observations of $V$. Therefore, estimating $f_\HH$ reduces to inverting $\LS I_k$. 
For this reason, 
it might also be of interest to consider the discrepancy 
$\Dis_k(f,f_*):=\left\|\LS f-b\right\|_{\HH}$.
We now present a few examples of objects related to $V$ that can be estimated in this framework. They all correspond to Setting (S1). 
We write $X\sim q$ to mean that the random variable $X$ has density function $q$. 
\subsubsection{Density estimation}
\label{sec density} 
Given a random vector $X\sim q_X$ in $\R^d$, we estimate its density $q_X$ on a subset $\mathcal{Z}\subset\R^d$. In this case, we have $V=X$ and $f_*=q_X\big|_{\mathcal{Z}}$. We choose $Z\sim q_Z$ as any random vector such that $q_Z$ has full support $\mathcal{Z}$. 
Assumption (A3) holds with $\op f =f q_Z$ because 
$$\E[f(Z)q_X(Z)]
=\E[f(X)q_Z(X)]$$
for all $f\in L^2(Z)$. 
Taking $f=k(z,\cdot)$, we obtain 
$b(z)=\E[k(z,Z)q_X(Z)]=\E[k(z,X)q_Z(X)]$ 
and we compute
$$
\left\|\op_2 k(\cdot,x)\right\|_{\HH}^2
=q^2_Z(x) k(x,x).
$$
If $q_Z$ is a uniform density, $b$ is called the kernel mean embedding of $X$. We deduce that Assumption (A4) holds with $C= \kappa^2 \left\| q_Z\right\|_\infty^2$. 
If $\mathcal{Z}$ contains the support of $X$, $\Dis$ is the integrated mean square error 
and $\Dis_k$ 
is called the kernel discrepancy or the maximum mean discrepancy. 
There is an extensive literature using $\Dis_k$ to design estimators but mostly in a parametric framework (see e.g.\ \cite{oates}). 

\subsubsection{Density ratio estimation}
Given two random vectors $X\sim q_X$ and $Y\sim q_Y$ such that the distribution of $X$ is absolutely continuous with respect to the distribution of $Y$, 
we can estimate the density ratio $q_X/q_Y$ from samples of $X$ and $Y$. In this case, we have $V=X$, $f_*=q_X/q_Y$ and we take $Z=Y$. We have 
$$\E\left[f(Y)\frac{q_X(Y)}{q_Y(Y)}\right]=\E[f(X)]$$
for all $f\in L^2(Y)$. Again, $b=\E[k(\cdot,X)]$ is the kernel mean embedding of $X$. Consequently, we can choose $\op f=f$ and we deduce that $C= \kappa^2$. 
The application of regularisation to density ratio estimation has been explored in \cite{Kanamori, Que, Schuster, Nguyen}. 

\subsubsection{Score estimation}
Some density estimation methods, such as maximum likelihood estimation, have the drawback that they require computing the integration constant of the density, which can be costly in high dimension.
One way of circumventing this issue is to estimate the score instead of the density, namely $\nabla \log q_X$. 
This score 
allows one to recover its associated density by integration. 
In this case, we have $V=X\sim q_X$, $f_*=\nabla \log q_X$ and $Z=X$. Integration by parts yields
$$\E[f(X)\cdot \nabla \log q_X(X)]=-\E[\mathrm{div} f(X)]$$
for all differentiable functions $f\in L^2(X)^d$ where $u\cdot v$ is the Euclidean scalar product of $u,v\in \R^d$ and $\mathrm{div}$ is the divergence operator. 
Hence, we have $\op f=-\mathrm{div}  f$.     
Here, $\Dis$ is the Fisher divergence and $\Dis_k(f,f_*)$ coincides with the kernelised Stein discrepancy between $q_X$ and the density associated to the score $f$. This discrepancy, which is closely related to Stein's method, has received significant attention in the statistical literature (see e.g. \cite{Barp}). 
The combination of score estimation and regularisation has been explored by \cite{Sriperumbudur, Zhou}. Since $f_*$ is a vector valued function, this case does not completely suit our framework 
but the generalisation is straightforward and can be found in \cite{Zhou}.

\subsubsection{Regression}
Given a vector of predictive variables $X$ and an outcome variable $Y$ such that $Y\in [-B,B]$ for some $B>0$, we can estimate the conditional expectation of $Y$ given $X$. We have $V=(X,Y)$, $f_*(x)=\E[Y|X=x]$ and we choose $Z=X$. In this case, Assumption (A1) always holds by definition of the conditional mean. We have by the tower property
$$\E[f(X)\E[Y|X]]=\E[f(X)Y]$$
for all $f\in L^2(X)$. 
Hence, we can choose $\op f(x,y)=f(x)y$ and we deduce that $C=\kappa^2 B^2$, but the assumption that $Y$ is bounded can be weakened (see e.g.\ Equation 1 in \cite{Bauer}). 
The application of regularisation to regression have been studied in \cite{DeVito, DeVito06, Caponnetto07, Bauer, Verri, Blanchard}. 
\subsection{Regularisation}
The operator $\LS I_k$ is Hilbert-Schmidt and thus compact  (see e.g.\ Equation (15) in  \cite{Caponnetto07}), which implies that its eigenvalues converge to 0. 
Consequently, the equation $\LS f_\HH=b$ is ill-posed in the sense that small perturbations in $b$ may lead to arbitrarily large variations in $f_{\HH}$.  
For this reason, the inversion of $\LS I_k$ must be regularised, which we accomplish using the general regularisation scheme (GRS). In the remainder of this section, some technical details are deferred to Appendix \ref{sec proof}.
We recall the definition of a regulariser, also known as a filter function, which constitutes the core element of GRS.
\begin{definition}
The family of functions $g_\lambda:[0,\kappa^2]\to\R$, $\lambda\in (0,\kappa^2)$, is called a regulariser  if there exist $B,D,\gamma>0$ such that  $\sup_{0<\sigma\leq\kappa^2}\left|\sigma g_\lambda(\sigma)\right|<D$, $\sup_{0<\sigma\leq\kappa^2}\left|g_\lambda(\sigma)\right|<B/\lambda$ and $\sup_{0<\sigma\leq\kappa^2}\left|1-\sigma g_\lambda(\sigma)\right|<\gamma$ for all $0<\lambda\leq\kappa^2$. 
The qualification of $g_\lambda$ is the maximal $\nu>0$ such that there exists $\gamma_\nu>0$ satisfying for all $\lambda\in (0,\kappa^2)$
$$
\sup_{0<\sigma\leq\kappa^2}\left|1-\sigma g_\lambda(\sigma)\right|\sigma^\nu\leq \gamma_\nu \lambda^\nu.
$$ 
\end{definition}
Let $A:\HH\to \HH$ be a compact self-adjoint operator with eigenvectors $(e_i)_{i=1}^m$ and eigenvalues $(\sigma_i)_{i=1}^m\subset [0,\kappa^2]$ with $m\in \N\cup\left\{\infty\right\}$. The spectral theorem for compact self-adjoint operators tells us that $(e_i)_{i=1}^m$ can be completed to form a basis $(e_i)_{i=1}^\infty$ of $\HH$. Given a regulariser $g_\lambda$, we define the operator $g_\lambda(A)$ by 
$$
g_\lambda(A)f=\sum_{i=1}^m g_\lambda(\sigma_i)f_i e_i
$$
where $f=\sum_{i=1}^\infty f_i e_i\in \HH$. The operator $g_\lambda(A)$ can be viewed as a regularised inverse of $A$, as will be illustrated by the examples below. 

In the sequel, we will use the same notation $\LS$ for the operators $\LS:L^2(Z)\to \HH$ and $\LS I_k:\HH\to\HH$, the domain should be clear from the context. 
The operator $\LS:\HH\to \HH$ is compact, as mentioned above, and self-adjoint because
\begin{equation}
\label{eq adjoint}
\left\langle \LS f,g\right\rangle_\HH=\E\left[f(Z)\left\langle  k(\cdot,Z),g\right\rangle_\HH\right]=\E[f(Z)g(Z)].
\end{equation}
We denote the operator norm on $\HH$ by $\left\|\cdot\right\|:=\left\|\cdot\right\|_{\HH\to \HH}$. One can show that $\left\|\LS\right\|\leq \kappa^2$ and thus the eigenvalues of $\LS$ are included in $[0,\kappa^2]$ (see Appendix \ref{sec proof}). 
Let $(v_i)_{i=1}^n$ be an i.i.d.\ sample of $V$ and let $(z_i)_{i=1}^{n_z}$ be observations of $Z$ satisfying Setting (S1) or (S2). 
%
The sample versions of $\LS$ and $b$ 
are defined by
\begin{equation}
\label{eq hLS}
\hat{\LS}:\HH\to \HH:f\mapsto \frac{1}{n_z}\sum_{i=1}^{n_z} k(\cdot,z_i)f(z_i)
\end{equation}
and
\begin{equation}
\label{eq hb}
\hat{b}=\frac{1}{n}\sum_{i=1}^n \op_2 k(\cdot,v_i).
\end{equation}
The GRS estimator of $f_\HH$ is defined by
\begin{equation}
\label{def estimator}
\hat{f}_\lambda=g_\lambda(\hat{\LS})\hat{b}.
\end{equation}
It can also be written in the following way 
\begin{align}
\label{eq fGRS}
\hat{f}_\lambda&=f_0+\sum_{i=1}^{n_z} \alpha_{i}k\left(\cdot,z_i\right)+\beta n\hat{b}
\end{align}
where $f_0\in \HH$, $\beta\in\R$ and $\alpha=(\alpha_{i})_{i=1}^{n_z}\in \R^{n_z}$ depend on the chosen regulariser. 
Before introducing the Tikhonov and Landweber regularisers, we observe that Assumption (A3) implies
\begin{equation}
\label{eq D}
\left\|f-f_*\right\|^2_{Z}=\left\|f\right\|^2_{Z}-2\E[\op f(V)]+\left\|f_*\right\|^2_{Z}. 
\end{equation}
Therefore, minimising $\left\|f-f_*\right\|_Z^2$ is equivalent to minimising $D(f)=\left\|f\right\|^2_{Z}-2\E[\op f(V)]$ 
which can be estimated by 
\begin{equation}
\label{eq hD}
\hat{D}(f)=\frac{1}{n_z}\sum_{i=1}^{n_z} f(z_i)^2-\frac{2}{\nn} \sum_{i=1}^{\nn} \op f(v_i).
\end{equation}

\subsubsection{Tikhonov regularisation}
\label{sec Tikhonov}
The Tikhonov regulariser is given by $g_\lambda(\sigma)=1/(\sigma+\lambda)$. In this case, we have $B=D=\gamma=1$, $\gamma_\nu=1$ for $\nu\in [0,1]$ and the qualification is equal to 1 (see e.g.\  \cite{Bauer}). Given an operator $A$, we have $g_\lambda(A)=(A+\lambda I)^{-1}$ where $I$ is the identity operator. 
One can show that $\hat{f}_\lambda$ 
is the solution of $\min_{f\in \HH} \hat{D}(f)+ \lambda\left\|f\right\|^2_{\HH}$ 
(see e.g.\ Proposition 3 in \cite{Caponnetto05}). 
Using the Representer Theorem (see e.g.\ Theorem A.2 in \cite{Sriperumbudur}), we obtain that $\hat{f}_\lambda$ is given by Equation \eqref{eq fGRS} with $f_0=0$, $\beta={2}/{(\nn\lambda)}$ and
\begin{equation}
\label{eq alpha}
\alpha=-\frac{4}{\lambda}\left(2 K+n_z\lambda I_{n_z}\right)^{-1} \hat{B}
\end{equation}
where 
$K=\left[k(z_i,z_j)\right]_{i,j=1}^{n_z}$, $I_{m}$ is the identity matrix of dimension $m\times m$ for $m\in \N$ and $\hat{B}=\left(\hat{b}\left(z_i\right)\right)_{i=1}^{n_z}$ 
(see Appendix \ref{sec proof} for a proof).


\subsubsection{Landweber iteration}
\label{sec Landweber}
The Landweber regularisator is given by
$$
g_t(\sigma)=\frac{1}{\sigma}\left\{1-\prod_{j=0}^{t-1}(1-\delta_j\sigma)\right\}
$$
where $t=\lfloor 1/\lambda\rfloor$ 
and $(\delta_i)_{i\in\N}\subset [0,1/\kappa^2]$. 
In this case, we have $B=1/\kappa^2$, $D=\gamma=1$,  $\gamma_\nu=1$ for $\nu\in [0,1]$ and $\gamma_\nu=\nu^\nu$ for $\nu>1$ (see \cite{Yao}). Hence, the qualification is infinite. 
To recall, the Fréchet gradient of $F(f)=\left\|f-f_*\right\|^2_{Z}$ on $\HH$ is $\nabla F(f)=2(\LS f-b)$. 
Minimising $\hat{D}(f)$ by gradient descent with the sequence of step size $(\delta_i)_{i\in\N}$ starting from $f_0\in \HH$ leads after $t$ steps to 
$$
\hat{f}_t=
\left\{I-\hat{\LS} g_t(\hat{\LS})\right\}f_0+g_t(\hat{\LS})\hat{b}
$$ 
where $I$ is the identity operator (see Proposition 4.2 in \cite{Yao}). Thus, $\hat{f}_t$ is of the form \eqref{def estimator} only if we take $f_0=0$. However, if $f_0=\LS^{\nu}u$ for some $\nu>0$ and $u\in L^2(Z)$, we have 
$$
\left\|\left(I-\hat{\LS} g_t(\hat{\LS})\right)f_0\right\|_Z\leq \left\|u\right\|_Z \gamma_\nu t^{-\nu}
$$
for all $t\in\N$ since the qualification of $g_t$ is infinite. We can express $\hat{f}_t$ in the form of Equation \eqref{eq fGRS} by taking $f_0$ as the initial function of the gradient descent, setting $\beta^{(t)}=-2/n \sum_{i=0}^{t-1}\delta_i$ and defining $\alpha$ through the following recurrence relation $\alpha^{(0)}=0$,  
$$
\alpha^{(t+1)}=\alpha^{(t)}-\frac{2\delta_t}{n_z} \left[F_0+K \alpha^{(t)}+n\hat{B}\beta^{(t)}\right]
$$
where $F_0=(f_0(z_i))_{i=1}^{n_z}$. 

The Landweber method has a numerical advantage compared with the Tikhonov one. 
In both cases, the estimator must be computed for each value of the regularisation parameter, namely $\lambda$ or $t$. The Tikhonov method requires inverting $2 K+n_z\lambda I_{n_z}$ for each value of $\lambda$. In contrast, the Landweber method only requires computing $K \alpha^{(t)}$ 
for each value of $t$. Therefore, the Landweber method replaces matrix inversions by matrix multiplications, which reduces the computational burden. 

\subsubsection{Exact line search}
\label{subsec ELS}
A key issue in the Landweber method concerns the choice of the step sizes $(\delta_i)$, also referred to as learning rates, in the gradient descent procedure. The Landweber method defines a valid regulariser if $\delta_i \leq 1/\kappa^2$ for all $i \in \mathbb{N}$ (see \cite{Yao}). While excessively large step sizes may compromise convergence, overly small step sizes require a larger number of iterations and  increase the computational cost. 
While $\delta_i = 1/\kappa^2$ is a natural default choice, it may result in unnecessarily conservative step sizes. 
Exact line search offers a way to adaptively select step sizes and potentially achieve a better balance.  
Because the objective function $F$ is quadratic, the optimal step size along the descent direction $-\nabla F(f)$---that is, the solution to $\min_{\delta>0} F(f - \delta \nabla F(f))$--- has an explicit expression 
\[
\delta_{}^{(1)}
=\frac{\left\| \LS f-b\right\|^2_{\HH}}{2\left\|\LS f-b\right\|^2_{Z}},
\]
which can be estimated. Alternatively, one may consider minimizing the discrepancy $\Dis_k(f-\delta \nabla F(f), f_*)^2$ with respect to $\delta$. The corresponding optimal step size is
\begin{equation}
\label{eq delta}
\delta^{(2)}
=\frac{\left\|\LS f-b\right\|^2_{Z}}{2\left\|\LS(\LS f-b)\right\|^2_{\HH}}
=\frac{\left\|\LS f-b\right\|^2_{Z}}{2\left\langle \LS(\LS f-b), \LS f-b\right\rangle_Z},
\end{equation}
where the second equality follows from Equation~\eqref{eq adjoint}. 
One can show that
\begin{equation}
\label{eq delta12}
\delta_{}^{(1)} \geq \delta_{}^{(2)} \geq \frac{1}{2\|\LS\|}
\end{equation}
for all $f \in \HH$ (see Appendix \ref{sec proof}). 
However, $\delta_{}^{(2)}$ admits no upper bound, and therefore neither $\delta_{}^{(1)}$ nor $\delta_{}^{(2)}$ satisfies the condition $\delta_i \leq 1/\kappa^2$. As a result, we cannot guarantee that $g_t$ defines a regulariser when step sizes are selected via exact line search, and the asymptotic results of Section~\ref{sec asymptotic} no longer apply in this setting. Showing that $g_t$ is still a regulariser when exact line search is employed is beyond the scope of the present paper and is deferred to future research. 

\subsection{Asymptotic properties}
\label{sec asymptotic}
We are interested in bounds of the form 
\begin{equation*}
\mathrm{P}\left[\left\|\hat{f}_\lambda-f_\HH\right\|_{L^2(Z)}<\epsilon(n) \log(1/\eta)\right]\geq 1-\eta
\end{equation*}
for some positive decreasing function $\epsilon$ and all $\eta \in (0,1)$. To obtain such a result, we need to impose regularity assumptions on $f_\HH$, the solution of \eqref{eq fH}. We call $\phi:[0,T]\to \R$ an index function if $\phi$ is continuous, strictly increasing, and satisfies $\phi(0)=0$. We say that $\phi$ is operator monotone if for all self-adjoint operators $U$ and $V$ on $\HH$ with spectra in $[0,T]$ such that $U\leq V$ it holds that $\phi(U)\leq \phi(V)$, where $U\leq V$ means that $\left\langle (U-V)h,h\right\rangle_\HH\leq 0$ for all $h\in \HH$. In \cite{Mathe}, it is proved that any $f\in \HH$ can be written as $f=\phi(\LS)u$ for some index function $\phi$ and some $u\in \HH$. A common practice in the literature is to assume that there exist $u\in \HH$ and an index function $\phi$, subject to additional assumptions, such that $f_\HH=\phi(\LS)u$, which is referred to as a source condition (see e.g.\ Equation (11) in \cite{Bauer}). 
One option frequently used  is to require that $\phi$ can be written as the product of an operator monotone function and a Lipschitz continuous function. To make the discussion simpler, we will enforce the stronger but more readable hypothesis that $\phi(t)=t^r$ for some $r>0$, which is called a Hölder source condition. In other words, we assume that 
\begin{itemize}
\item[(A5)] there exist $r>0$ and $u\in\HH$ such that $f_\HH=\LS^r u$.
\end{itemize}
Before stating the main theorem, we give two technical results 
(see Appendix \ref{sec proof} for the proofs).
\begin{lemma}
\label{lem 1}
Let Assumptions (A2) and (A4) hold, $\nn\in\N$ and $\eta\in (0,1)$. Then,  
 with probability at least $1-\eta$, it holds 
$$
\left\|\hat{b}-b\right\|_{\HH}\leq \frac{6C}{\sqrt{\nn}}\log \frac{2}{\eta}.
$$
where $C$ is given in Assumption (A4), $b$ in Equation \eqref{eq b} and $\hat{b}$ in Equation \eqref{eq hb}. 
\end{lemma}

\begin{lemma}
\label{lem 2}
Let Assumption (A2) hold, $n_z,\nn,n_u\in\N$ and $\eta\in (0,1)$. Then, with probability at least $1-\eta$, it holds under (S1) 
$$
\left\|\hat{\LS}-\LS\right\|\leq \frac{6\kappa^2}{\sqrt{n_z}}\log\frac{2}{\eta} 
$$
and it holds under (S2) 
$$
\left\|\hat{\LS}-\LS\right\|\leq \left(\frac{6\kappa^2}{\sqrt{\nn}}+\frac{6\kappa^2}{\sqrt{n_u}}\right)\log\frac{4}{\eta} 
$$
where $\kappa^2$ is given in Assumption (A2), $\LS$ in Equation \eqref{eq LS} and $\hat{\LS}$ in Equation \eqref{eq hLS}.
\end{lemma}

We now state the asymptotic properties of the GRS estimator, which are obtained by a straightforward adaptation of Theorem 10 in \cite{Bauer} (see Appendix \ref{sec proof} for a proof).
\begin{theorem}
\label{theo main}
Consider either Setting (S1) with $n_z=m$ or (S2) with $n_u=m$. Let Assumptions (A1) to (A5) hold for some $r>0$ and $u\in \HH$. Let $\bar{r}$ be the qualification of the regulariser $g_\lambda$ and $\hat{f}_\lambda$ be defined as in \eqref{def estimator}. 
Let $\eta\in (0,1)$,  
choose $\lambda=(n^{-1/2}+m^{-1/2})^{1/(r+1)}$ and assume that 
$\lambda^{-r}>6\kappa^2\log(4/\eta)$. 
Then, with probability at least $1-\eta$, 
if $r\in [0,\bar{r}]$ it holds  
$$
\left\|\hat{f}_\lambda-f_\HH\right\|_{\HH}\leq C_1  \left(\frac{1}{\sqrt{n}}+\frac{1}{\sqrt{m}}\right)^{\frac{r}{r+1}}\log\left(\frac{4}{\eta}\right)
$$
and if $r\in [0,\bar{r}-1/2]$ it holds
$$
\left\|\hat{f}_\lambda-f_\HH\right\|_{Z}\leq C_2 \left(\frac{1}{\sqrt{n}}+\frac{1}{\sqrt{m}}\right)^{\frac{2r+1}{2r+2}}\log\left(\frac{4}{\eta}\right)
$$
where $C_1$ and $C_2$ are constants that do not depend on $n$, $m$ and $\eta$.
\end{theorem}

\begin{remark}
Setting $m=n$ in Theorem \ref{theo main} yields convergence rates of order $ \nn^{-{r}/{(2r+2)}}$ in $\HH$ and $\nn^{-{(2r+1)}/{(4r+4)}}$ in $L^2(Z)$ for $\hat{f}_\lambda$. These rates can be improved if we add assumptions on the capacity of the hypothesis space, usually measured by the effective dimension 
$$
\mathcal{N}(\lambda):=\mathrm{Tr}\left((\LS+\lambda I)^{-1}\LS\right),\quad \lambda>0.
$$
For example, under the assumption that $\mathcal{N}(\lambda)\leq c \lambda^{-\beta}$ for some $\beta\in (0,1]$ and $c>0$, Corollary 5.1 in \cite{Lu} states that 
$$
\left\|\hat{f}_\lambda-f_\HH\right\|_{\HH}\leq C_1 \nn^{-\frac{r}{2r+1+\beta}}\left(\log\frac{6}{\eta}\right)^2
$$
and 
$$
\left\|\hat{f}_\lambda-f_\HH\right\|_{Z}\leq C_2 \nn^{-\frac{2r+1}{4r+2+2\beta}}\left(\log\frac{6}{\eta}\right)^3
$$
with $\lambda=n^{-1/(2r+1+\beta)}$. 
\end{remark}

\section{Conditional density estimation}
\label{sec CDE}
Now we apply the general regularisation scheme (GRS) to conditional density estimation (CDE). 
Let $X$ be a predictive vector with values in a closed set $\mathcal{X}\subset \R^d$ and $Y$ be an univariate outcome variable. We consider an univariate response $Y$ to simplify discussion, but the generalisation to multivariate responses is straightforward. We denote by $q_{Y|X}$ the conditional density of $Y$ given $X$. We want to estimate $q_{Y|X=x}(u)$ for all $(x,u)\in \X\times \mathcal{U}$ where   $\mathcal{U}\subset \R$ is a closed set. We choose a random variable $U\sim q_U$, independent of $X$ and $Y$, such that $q_U$ is bounded and has full support $\mathcal{U}$. In the notation of Section \ref{sec GRS}, we set $V=(X,Y)\sim q_{(X,Y)}$, $Z=(X,U)\sim q_X q_U$ and $f_*(x,u)=q_{Y|X=x}(u)$ for $(x,u)\in\mathcal{Z}=\mathcal{X}\times \mathcal{U}$. 
Assumption (A3) holds because 
\begin{align}
\label{eq trick CDE}
\E\left[f\left(X,U\right)f_*(X,U)\right]
=\E\left[f\left(X,Y\right)q_U(Y)\right]
\end{align}
where we used the fact that $q_{(X,Y)}=q_{Y|X}q_X$. 
We can thus choose $\op f(x,y)=f(x,y)q_U(y)$ 
and Assumption (A4) holds with $C=\kappa^2\left\|q_U\right\|_\infty^2$. 

Let  $((x_i,y_i))_{i=1}^{n}$ and $(u_j)_{j=1}^{n_u}$ be i.i.d.\  samples of $(X,Y)$ and $U$ respectively. We define $z_i$ as in Equation \eqref{eq zi} for $i=1,\ldots,n_z$ with $n_z=n_un$, 
which places us in Setting (S2). Alternatively, the problem could be formulated under Setting (S1) by taking $z_i=(x_i,u_i)$ for $i=1,\ldots,n$ but this would lead to a different estimator. We adopt Setting (S2) because, as explained in Appendix \ref{rem two estimators}, $\hat{\LS}$ has smaller variance under this setting. 
We recall that our GRS-based estimator is defined in Equation \eqref{eq fGRS}. 
As explained in Section \ref{sec Landweber}, the Landweber method has some numerical advantages compared to the Tikhonov one, which are even more acute in the case of CDE if we assume that 
the kernel is of the form 
$$k((x,y),(x',y'))=k_X(x,x')k_Y(y,y')$$ where $k_X:\R^d\times\R^d\to \R$ and $k_Y:\R\times\R\to \R$.  
Indeed, under this assumption, the matrix product with $K$ can be written as
$$
K\left(f(z_i)\right)_{i=1}^{n_z}=K_X \left[f(x_i,u_j)\right]_{i,j=1}^{\nn,n_u}K_U
$$
where 
$K_X=[k_X(x_i,x_l)]_{i,l=1}^{\nn}$ and $K_U=[k_Y(u_j,u_l)]_{j,l=1}^{n_u}$. 
Hence, instead of multiplying by $K$ (which requires $\nn^2n_u^2$ multiplications and additions), we can multiply by $K_X$ and $K_U$ (which requires only $(\nn+n_u)\nn n_u$). 
This simplification can also be reached with the Tikhonov method by observing that, since $X$ and $U$ are independent, $\LS$ can be written as 
\begin{align*}
\LS f(x,y)&=\E\left[k_X(x,X)\E[k_Y(y,U)f(X,U)|X]\right]\\
&=\LS_X\LS_U f(x,y)
\end{align*}
with $\LS_X g(x)=\E\left[k_X(x,X)g(X)\right]$ and $\LS_U g(y)=\E\left[k_Y(y,U)g(U)\right]$. This allows to invert $\LS_X$ and $\LS_U$ separately but it introduces the drawback that we have to select two regularisation parameters, one for the inversion of $\LS_X$ and one for the inversion of $\LS_U$. 
For this reason and those mentioned in Section \ref{sec Landweber}, we will implement the Landweber regularisation in our numerical experiments. 
 
\subsection{Kernel estimators}
We will compare our estimator, that we call GRS, with three other kernel CDE which we briefly describe. Recall that GRS is based 
on the following representation of the conditional density
$$
q_{Y|X=x}(y)=\LS^{-1}_U\LS_X^{-1} b(x,y)
$$
where 
\begin{align*}
b(x,y)&=\E\left[k_X(x,X)k_Y(y,Y)q_U(Y)\right]\\
&=\E\left[k_X(x,X)k_Y(y,U)p_{Y|X}(U)\right]
\end{align*}
is the kernel mean embedding of $(X,Y)$. The kernel conditional density operator estimator (CDO) of \cite{Schuster} uses another representation of the conditional density, namely
$$
q_{Y|X=x}(y)=\LS_U^{-1} \left[\LS_{XY}\LS_X^{-1} k_X(\cdot,x)\right](y)
$$
where $\LS_{XY}f(y)=\E[k_Y(y,Y)f(X)]$ and in particular $b=\LS_{XY}k_X$. It is very similar to GRS, differing only by the permutation of $\LS_{XY}$ and $\LS_X^{-1}$. CDO can be written as 
    $$
    \hat{f}_{\mathrm{CDO}}(x,y)=\sum_{j=1}^{n_u} \tilde{w}_j(x) \ k_Y\left(y,u_j\right). 
    $$
where $\tilde{w}(x)=1/n_u^2 (K_U +\lambda_2 I_{n_u})^{-2}K_{U,Y}w(x)$, $w(x)=(K_{X} +\nn\lambda_1 I_{\nn})^{-1}K_{x,X}$, $K_{x,X}=(k(x,x_i))_{i=1}^{\nn}$ and $K_{U,Y}=[k_Y(u_j,y_i)]_{j,i=1}^{n_u,\nn}$.
No convergence rate is provided for CDO in \cite{Schuster}. 

The kernel mean density estimator (KMD) of \cite{Spiteri} relies on the following representation of the conditional density
$$
q_{Y|X=x}(y)=\lim_{h_{Y}\to 0}\frac{1}{q_U(y)}  \LS_{XY}\left[\LS_X^{-1} k_X(\cdot,x)\right](y) 
$$
where $h_Y>0$ is the bandwidth parameter of $k_Y$ and satisfies $\lim_{h_Y\to 0} k_Y(u,y)=\delta_y(u)$. KMD can be written as
    $$
    \hat{f}_{\mathrm{KMD}}(x,y)=\sum_{i=1}^{\nn} w_i(x) \ k_Y\left(y,y_i\right). 
    $$
Consistency of KMD in the supremum norm is established in Lemma 2.2 of \cite{Spiteri} but this result does not provide any convergence rate. 
 Since 
$$\lim_{h_{Y}\to 0}\LS_U f(y)=\lim_{h_{Y}\to 0}\E[k_Y(y,U)f(U)]=f(y)q_U(y),$$ 
we have 
$\lim_{h_{Y}\to 0}\LS_U^{-1} f(y)=f(y)/q_U(y)$ for all $f\in L^2(U)$. We conclude that
\begin{align*}
&\lim_{h_{Y}\to 0}\LS^{-1}_U \left[\LS_{XY}\LS_X^{-1} k_X(\cdot,x)\right](y)\\
&\quad=\lim_{h_{Y}\to 0}\frac{1}{q_U(y)}\LS_{XY}\left[\LS_X^{-1} k_X(\cdot,x)\right](y)
\end{align*}
which shows that KMD can be viewed as a limiting case of CDO as $h_{Y}\to 0$. 

Finally, the Nadaraya-Watson estimator (NW) is obtained by replacing $q_{(X,Y)}$ and $q_{X}$ with their respective  kernel density estimators in the identity $q_{Y|X}=q_{(X,Y)}/q_X$. This yields 
$$
\hat{f}_{\mathrm{NW}}(x,y)=\sum_{i=1}^{\nn}\frac{k_X\left(x,x_i\right)}{\sum_{j=1}^{\nn} k_X\left(x,x_j\right)}k_Y\left(y,y_i\right). 
$$
The convergence rate of NW with respect to the integrated mean square error (see Equation \eqref{eq MSE} below) is of order $n^{-2/3}$ when $d=1$ (see \cite{Hyndman}). 

Remark that only $\hat{f}_{\mathrm{NW}}(x,\cdot)$ is guaranteed to be nonnegative and to integrate to one. The other three estimators can be normalised by applying the ReLU function and dividing by the integral with respect to $y$. 

\subsection{Numerical experiments}
We now conduct some numerical experiments on synthetic data. 
We will compare the performances of GRS, NW and KDE. We do not take CDO into account since it is very similar to GRS and KMD and requires two regularisation parameters. We implement GRS with Landweber regularisation and the step sizes $\delta_t=1/\kappa^2$ and $\delta_t=\delta^{(2)}$ as described in Equation \eqref{eq delta}.  We use for each estimator the Gaussian kernels 
$$
k_X(x,x')=
\exp\left(-\frac{1}{2}(x-x')^{\top}H^{-1}(x-x')\right)
$$
and
$$
k_Y(y,y')=\frac{1}{h_{Y}\sqrt{2\pi}}\exp\left(-\frac{\left|y-y'\right|^2}{2h_{Y}^2}\right)
$$
where $H=\mathrm{Diag}(h_{1}^2,\ldots,h_{d}^2)$, $h_X=(h_1,\ldots,h_{d})$ is the vector of input bandwidths and $h_Y$ is the output bandwidth. 
Taking inspiration from the median heuristic (see \cite{Garreau}), 
we select $h_X$ 
from the set $\left\{M_{X}p_X^{l}:l=-L_X,\ldots,L_X\right\}$  where $p_X>1$ and $L_X\in\N$ are parameters and $M_X=(M_{l})_{l=1}^d$ with
$$M_{l}=\sqrt{\mathrm{med}(|x_{i,l}-x_{j,l}|^2,1\leq i< j\leq n)/2}$$ 
for $l=1,\ldots,d$ where $x_i=(x_{i,1},\ldots,x_{i,d})$. 
The output bandwidth $h_{Y}$ is selected from the set $\left\{M_Yp_Y^{l}:l=-L_Y,\ldots,L_Y\right\}$  where $p_Y>1$ and $L_Y\in\N$ are parameters and
$$M_{Y}=\sqrt{\mathrm{med}(|y_{i}-y_{j}|^2,1\leq i< j\leq n)/2}.$$
The regularisation parameter $\lambda$ is chosen in the set $\left\{p_\lambda^{-l}:l=0,1,\ldots,L_\lambda\right\}$ where $p_\lambda>1$ and $L_\lambda\in\N$ are parameters. We take 
$q_U$ as the uniform density on $\mathcal{U}$ and we choose $\mathcal{U}$ on a case-by-case basis. 
For the GRS estimator, we start the gradient descent from $f_0=q_U$ and we do at most $T_1=40$ iterations with $\delta_t=1/\kappa^2$ and $T_2=10$ iterations with $\delta_t=\delta^{(2)}$. 
For each estimator, we proceed in the following way: 
\begin{itemize}
\item We construct an estimator for each value of $(h_X,h_Y, t, \lambda)$ with a training sample of $(X,Y)$ of size $n_{\mathrm{train}}$. 
\item We select the optimal value of $(h_X,h_Y, t, \lambda)$ as the minimiser of $\hat{D}(\hat{f})$ (see Equation \eqref{eq hD}) evaluated with a validation sample of $(X,Y)$ of size $n_{\mathrm{val}}$. We recall that $\hat{D}(\hat{f})$  is, up to an additive constant, an estimator of $\E[|\hat{f}(X,U)-q_{Y|X}(U)|^2]$.
\item We compute the integrated mean square error 
\begin{equation}
\label{eq MSE}
\mathrm{MSE}(\hat{f})=\frac{1}{\nn n_u}\sum_{i=1}^n\sum_{j=1}^{n_u}\left|\hat{f}(x_i,u_j)-q_{Y|X=x_i}(u_j)\right|^2
\end{equation}
for the chosen value of $(h_X,h_Y, t, \lambda)$ with a test sample of $(X,Y)$ of size $n=n_{\mathrm{test}}$. 
\end{itemize}
We repeat the experiment $n_{\mathrm{MC}}$ times and we report in our tables the mean and standard deviation of MSE$(\hat{f})$ over the $n_{\mathrm{MC}}$ repetitions. 

\subsubsection{Mixture of Gaussian}
We consider a mixture of Gaussian densities with means located on a circle (see \cite{Schuster}). To be precise, 
we first draw a discrete uniform variable over the set $\left\{(0_{d-1},\cos(2\pi i/n_p),\sin(2\pi i/n_p))\right\}_{i=1}^{n_p}\subset\R^{d+1}$ with $n_p=50$ and then we draw an isotropic Gaussian variable with the first variable as mean. The resulting random vector is denoted by $W$ and we set $X=(W_1,\ldots,W_d)$ and $Y=W_{d+1}$. 
Let $\theta_i=(0_{d-1},\cos(2\pi i/n_p))$ for $i=1,\ldots,50$. The conditional density is 
$$
q_{Y|X=x}(y)=\sum_{i=1}^{n_p}\frac{q_{\mathcal{N}(\theta_i,I_{d})}(x)}{\sum_{i=1}^{n_p}q_{\mathcal{N}(\theta_i,I_{d})}(x)}q_{\mathcal{N}\left(\sin\left(\frac{2\pi i}{n_p}\right),1\right)}(y).
$$ 
We take $\mathcal{U}=[\min_{1\leq i\leq n}y_i,\max_{1\leq i\leq n} y_i]$. 
NW should be highly precise in this case since the true density $q_{Y|X}$ and NW share a similar structure. Despite this, we see from Table \ref{tab gaussian mixture} that GRS and NW exhibit similar accuracy. 

\begin{table}
\caption{Mean (first row) and standard deviation (second row) of the MSE$(\hat{f})$, both multiplied by $10^3$, for the mixture of Gaussian across $n_{\mathrm{MC}}=100$ replications 
with $n_{\text{train}}=100$, $n_{\text{val}}=100$, $n_{\text{test}}=100$, $n_u=50$, $p_X=2$, $p_Y=1.6$, $p_\lambda=3$, $L_X=3$, $L_Y=3$, $L_\lambda=6$, $T_1=40$ and $T_2=10$.} 
\label{tab gaussian mixture}
  \begin{center}
    \begin{small}
      \begin{sc}
\begin{tabular}{c c c c c}
\toprule
$d$  & $\delta^{(2)}$ & $1/\kappa^2$ & NW & KMD  \\
\midrule
2  & 1.02 & 1.00 & 1.22 & 1.56 \\
   & 0.913 & 0.776 & 0.915 & 1.03 \\
\midrule
6  & 1.12 & 1.00 & 1.07 & 1.83 \\
   & 1.32 & 0.750 & 0.866 & 0.989 \\
\midrule
10 & 1.03 & 1.04 & 1.06 & 2.03 \\
   & 0.765 & 0.717 & 1.03 & 1.09 \\
\bottomrule
\end{tabular}
      \end{sc}
    \end{small}
  \end{center}
  \vskip -0.1in
\end{table}

\subsubsection{Cox–Ingersoll–Ross model}
\label{sec CIR}
In the Cox–Ingersoll–Ross model (CIR), 
a short-term interest rate $X_t$ evolves according to the following equation
$$\mathrm{d}X_t=\mu(\theta-X_t)\mathrm{d}t+\sigma \sqrt{X_t}\mathrm{d}W_t$$
where $W_t$ is a Wiener process and $(\mu,\theta,\sigma)$ are parameters (see \cite{Fu11}). We take  $X_0\sim \mathrm{\Gamma}(2\mu\theta/\sigma^2,\sigma^2/(2\mu))$, which is the invariant distribution of the process. We simulate the process at a monthly frequency, which means that the interval time is $\mathrm{d}t=1/12$. We have
$$X_{t+1}|X_t\sim \frac{(1-e^{-\mu\mathrm{d}t})\sigma^2}{4\mu}S
$$
where $S$ is a non-central chi-squared variable with $\frac {4\mu\theta}{\sigma ^{2}}$ degrees of freedom and non-centrality parameter $2cX_{t}e^{-\mu \mathrm{d}t}$. 
We repeatedly generate a sample set of $m=n_{\mathrm{train}}+n_{\mathrm{val}}+n_{\mathrm{test}}$ monthly observations $(x_i)_{i=1}^{m}$. We estimate $q_{X_{t+1}|X_t}$ from the dataset $\left\{(x_i,x_{i+1})\right\}_{i=1}^{m-1}$ which is randomly split into training, validation, and test sets.  
We choose $(\mu,\theta,\sigma)=(0.21459, 0.08571, 0.0783)$ as in \cite{Fu11} and $\mathcal{U}=[0,0.3]$. 
The results of Table \ref{tab CIR} 
show that KMD has lowest MSE while GRS and NW demonstrate comparable performance.
In this case, the step size $1/\kappa^2$ is substantially smaller than $\delta^{(2)}$, resulting in a  higher MSE despite $T_1>T_2$. Achieving convergence with the step size $1/\kappa^2$ would require a much larger number of gradient descent iterations.
\begin{table}
\caption{Mean (first row) and standard deviation (second row)  of MSE$(\hat{f})$ for the CIR model across $n_{\mathrm{MC}}=100$ replications 
with $n_{\text{train}}=100$, $n_{\text{val}}=100$, $n_{\text{test}}=100$, $n_u=50$, $p_X=2$, $p_Y=1.6$, $p_\lambda=3$, $L_X=3$, $L_Y=3$, $L_\lambda=6$, $T_1=40$ and $T_2=10$} 
\label{tab CIR}
 \begin{center}
    \begin{small}
      \begin{sc}
\begin{tabular}{c c c c}
\toprule
   $\delta^{(2)}$ & $1/\kappa^2$ &NW &  KMD \\ 
\midrule

25.4 & 54.3 & 24.6 & 20.7 \\
8.90 & 12.0 & 6.59 & 6.51 \\
\bottomrule
\end{tabular}
     \end{sc}
    \end{small}
  \end{center}
  \vskip -0.1in
\end{table}  

\subsubsection{Autoregressive model}
We continue to investigate time series and consider an autoregressive model AR$(d)$
$$
X_t=\sum_{i=1}^d \phi_i X_{t-i}+\epsilon_{t}
$$ 
where $\phi _{1},\ldots ,\phi _{d}$ are parameters and $\epsilon _{t}$ is a standard Gaussian white noise. We thus have
$$
X_t|X_{t-1},\ldots,X_{t-d}\sim \sum_{i=1}^d \phi_i X_{t-i}+\mathcal{N}(0,1).
$$
We take $X_0\sim \mathcal{N}(0,4/3)$, $\phi_i=1/(2d)$ for all $i=1,\ldots,d$ and $\mathcal{U}=[\min_{1\leq i\leq n}y_i,\max_{1\leq i\leq n} y_i]$. 
Note that $X_0$ does not follow the stationary distribution of the process $(X_t)$ except for $d=1$. Nevertheless, for large $t$, the distribution of $X_t$ should be close to the stationary one. 
For this reason, we create the dataset as in Section \ref{sec CIR} but we generate 100 additional observations and discard the first 100 as a burn-in period. 
We see from Table \ref{tab MA} that GRS has the lowest MSE in high dimension. The advantage of $\delta^{(2)}$ over $1/\kappa^2$ appears to diminish as the dimension increases.

\begin{table}[t]
\caption{Mean (first row) and standard deviation (second row) of MSE$(\hat{f})$, both multiplied by $10^3$, for the AR$(d)$ model across $n_{\mathrm{MC}}=100$ replications 
with $n_{\text{train}}=100$, $n_{\text{val}}=100$, $n_{\text{test}}=100$, $n_u=50$, $p_X=2$, $p_Y=1.6$, $p_\lambda=3$, $L_X=3$, $L_Y=3$, $L_\lambda=6$, $T_1=40$ and $T_2=10$.} 
\label{tab MA}
 \begin{center}
    \begin{small}
      \begin{sc}
\begin{tabular}{c c c c c}
\toprule
$d$  &  $\delta^{(2)}$ & $1/\kappa^2$ & NW & KMD \\
\midrule
2  & 2.20 & 3.10 & 2.46 & 2.22 \\
   & 0.855 & 0.371 & 0.625 & 0.581 \\
\midrule
6  & 3.03 & 3.10 & 3.65 & 3.44 \\
   & 0.936 & 0.636 & 1.05 & 0.978 \\
\midrule
10 & 3.15 & 3.06 & 3.68 & 3.80 \\
   & 1.17 & 0.754 & 0.985 & 0.922 \\
\bottomrule
\end{tabular}
     \end{sc}
    \end{small}
  \end{center}
  \vskip -0.1in
\end{table}  

\subsubsection{Beta model}
To compare the estimators in a non-Gaussian setting, we consider a Beta model
$$
Y\sim \mathrm{Beta}\left(\alpha=1+\frac{1}{d}\sum_{i=1}^d X_i^2, \beta=1\right)
$$
where $X=\left(X_1,\ldots,X_d\right)\sim \mathrm{Unif}[0,1]$. We take $\mathcal{U}=[0,1]$. 
Table~\ref{tab beta} shows that GRS achieves the lowest MSE, followed by NW. In this example, the step size $1/\kappa^2$ yields the best performance.  
\begin{table}[b]
\caption{Mean (first row) and standard deviation (second row)  of MSE$(\hat{f})$, both multiplied by $10^2$, for the Beta model across $n_{\mathrm{MC}}=100$ replications 
with $n_{\text{train}}=100$, $n_{\text{val}}=100$, $n_{\text{test}}=100$, $n_u=50$, $p_X=2$, $p_Y=1.6$, $p_\lambda=3$, $L_X=3$, $L_Y=3$, $L_\lambda=6$, $T_1=40$ and $T_2=10$.} 
\label{tab beta}
 \begin{center}
    \begin{small}
      \begin{sc}
\begin{tabular}{c c c c c}
\toprule
$d$ &  $\delta^{(2)}$ & $1/\kappa^2$ & NW & KMD \\
\midrule
2  & 5.98 & 5.43 & 7.25 & 8.82 \\
   & 3.96 & 3.12 & 3.42 & 4.61 \\
\midrule
6  & 5.75 & 5.04 & 6.96 & 9.12 \\
   & 4.16 & 3.26 & 3.45 & 3.13 \\
\midrule
10 & 4.80 & 3.83 & 5.64 & 9.05 \\
   & 3.13 & 2.28 & 2.65 & 2.94 \\
\bottomrule
\end{tabular}
     \end{sc}
    \end{small}
  \end{center}
  \vskip -0.1in
\end{table}

\section{Conclusion}
We propose an abstract framework that encompasses several existing settings and allows the systematic application of the general regularisation scheme. Building on this, we develop a new conditional density estimator that enjoys well-established convergence rates, in contrast to kernel-based alternatives such as CDO or KMD for which such guarantees are not available. 
We contend that Landweber iteration is computationally more tractable than Tikhonov regularisation, especially with separable kernels, and that exact line search offers an efficient strategy for selecting the learning rate. 
Through numerical experiments, we demonstrate that our estimator performs comparably to NW in settings where the latter is particularly well suited—such as Gaussian mixtures—and that it can surpass NW in specific scenarios, notably for time series data. Finally, a natural direction for future research is to establish that the asymptotic properties of GRS still hold when exact line search is employed.



\section*{Acknowledgements}
I would like to thank Gesine Reinert and Adrian Fischer for reading my manuscript and their helpful comments. The author gratefully acknowledges support from the Wiener-Anspach Foundation.






\bibliography{biblio}
\bibliographystyle{icml2026}

\newpage
\appendix
\onecolumn

\section{Supplementary material}

\subsection{Settings for CDE}
\label{rem two estimators}
In Section \ref{sec CDE}, to estimate $\LS f(z)=\E[k(z,Z)f(Z)]$ for $f\in L^2(Z)$ and $z\in \R^{d+1}$, we have to estimate an expectation with respect to $Z=(X,U)$. Given a function $g\in L^1(Z)$, there are two possible estimators of $\E[g(Z)]$:
$$
\theta_1=\frac{1}{\nn n_u}\sum_{i=1}^{\nn}\sum_{j=1}^{n_u}g(x_i,u_j)\quad\text{and}\quad\theta_2=\frac{1}{\nn}\sum_{i=1}^{\nn} g(x_i,u_i).
$$
We have chosen $\theta_1$ because its variance is smaller when $n_u$ is large. Indeed, we have 
\begin{align*}
\Var[\theta_1]=\frac{1}{\nn}\left(\E\left[\left|\frac{1}{n_u}\sum_{j=1}^{n_u} g(X,U_j)\right|^2\right]-\E[g(X,U)]^2\right)
+\frac{\nn-1}{\nn n_u}\Cov[g(X,U),g(X',U)]
\end{align*}
where $X'$ is an independent copy of $X$. The first term is smaller than $\Var[g(X,U)]/\nn=\Var[\theta_2]$ by Jensen's inequality while the second term converges to 0 when $n_u\to\infty$. If instead we choose to implement $\theta_2$, then $(z_i)_{i=1}^n=((x_i,u_i))_{i=1}^n$ is an i.i.d.\ sample of $Z=(X,U)$ and the resulting CDE estimator falls under Setting (S1).

\subsection{Conditional density estimation on real data} 
We consider a dataset containing information about medical insurance charges\footnote{\label{foot:insurance}The dataset is available at \url{https://www.geeksforgeeks.org/machine-learning/dataset-for-linear-regression/}.}. The response variable $Y$  is the charges paid by an insurance company for an individual. The explanatory vector $X$ has five components:  age, sex, BMI, number of children, and smoker status. The dataset has size $n_\mathrm{data}=1338$. Again, we estimate the conditional density of $Y$ given $X$ as in Section \ref{sec CDE}. We take 
$\mathcal{U}=[\min_{1\leq i\leq n_\mathrm{data}}y_i,\max_{1\leq i\leq n_\mathrm{data}}y_i]$. 
Since $q_{Y|X}$ is not available, we replace $\mathrm{MSE}(\hat{f})$ by $\hat{D}(\hat{f})$. 
The results in Table \ref{tab insurance} show that GRS implemented with $\delta^{(2)}$ has the lowest mean $\hat{D}(\hat{f})$ but the highest variance. 
\begin{table}[b]
\caption{Mean (first row) and standard deviation (second row) of $D(\hat{f})$, both multiplied by $10^9$, for the insurance dataset across $n_{\mathrm{MC}}=100$ replications 
with $n_{\text{train}}=100$, $n_{\text{val}}=100$, $n_{\text{test}}=100$, $o=50$, $p_x=2$, $p_y=2$, $p_\lambda=3$, $L_x=3$, $L_y=3$, $L_\lambda=6$ and $T=5$.}
\label{tab insurance}
 \begin{center}
    \begin{small}
      \begin{sc}
\begin{tabular}{c c c c}
\toprule
$\delta^{(2)}$ & $1/\kappa^2$ & NW & KMD  \\    
\midrule
-2.94 &-1.19 &-1.90 &-2.16\\
1.49 &0.09 &0.27 &0.32\\
\bottomrule
\end{tabular}
     \end{sc}
    \end{small}
  \end{center}
  \vskip -0.1in
\end{table}

\section{Proofs}
\label{sec proof}

\begin{proof}[Proof that $\left\| \LS\right\|_{\HH\to\HH}=\left\| \LS\right\|_{L^2(Z)\to L^2(Z)}\leq  \kappa^2$]
Since $\LS$ is an integral operator with bounded symmetric kernel $k$, it is compact and self-adjoint on $L^2(Z)$. Hence, the spectral theorem for compact self-adjoint operators tells us that the eigenvectors $\left(\tilde{e}_i\right)_{i\in\N}$ of $\LS$ form an orthogonal basis of $L^2(Z)$. Moreover, the scalar product on $\HH$ of two elements $f=\sum_{i\in\N}f_i \tilde{e}_i$ and $g=\sum_{i\in\N}g_i \tilde{e}_i$ is given by $$\left\langle f,g \right\rangle_{\HH}=\sum_{i\in\N}\frac{f_i g_i}{\sigma_i}$$ 
where $\left(\sigma_i\right)_{i\in\N}$ are the eigenvalues of $\LS$ (see e.g. \cite{Liu}). Hence, an orthogonal basis of $\HH$ is given by $\left(e_i\right)_{i\in\N}$ with $e_i=\sqrt{\sigma_i}\tilde{e}_i$. 
We can conclude that $\left\| \LS\right\|_{L^2(Z)\to L^2(Z)}=\sigma_1=\left\| \LS\right\|_{\HH\to\HH}$ where $\sigma_1$ is the largest eigenvalue of $\LS$ and $\left\| \cdot\right\|_{L^2(Z)\to L^2(Z)}$ denotes the operator norm on $L^2(Z)$. 
Finally, using the Cauchy-Schwarz inequality and Assumption (A2), one obtains that  $\left\| \LS\right\|_{L^2(Z)\to L^2(Z)}\leq  \kappa^2$. 
\end{proof}

\begin{proof}[Proof of Equation \eqref{eq alpha}]
We define $\hat{D}:L^2(Z)\to\R$ and $V:\R^{n_z}\times \R\to\R$ by 
$$\hat{D}(f)=\frac{1}{n_z}\sum_{i=1}^{n_z} f(z_i)^2-\frac{2}{\nn} \sum_{i=1}^{\nn} \op f(v_i)\quad\text{and}\quad
V(a,b)=\frac{1}{n_z}\sum_{i=1}^{n_z} a_i^2-\frac{2}{\nn} b.$$
We have 
$$\hat{D}(f)=V\left(\left\langle f,k(\cdot,z_1) \right\rangle_{\HH},\ldots,\left\langle f,k(\cdot,z_n) \right\rangle_{\HH}, \left\langle f,\sum_{i=1}^{\nn} \op_2 k(\cdot, v_i) \right\rangle_{\HH}\right)$$
and $\nabla V(a,b)=(2a_1/n_z,\ldots, 2a_{n_z}/n_z,-2/\nn)$. Theorem A.2 in \cite{Sriperumbudur} tells us that the solution of 
$\inf_{f\in \HH}\hat{D}(f)+\left\|f \right\|_{\HH}^2$ is given by Equation \eqref{eq fGRS} with $f_0=0$ and $(\alpha,\beta)$ being the solution of 
$$
\lambda (\alpha,\beta)+ \nabla V(K\alpha+\beta n\hat{B}, \alpha^\top n\hat{B}+\beta1_\nn^\top \op_1\op_2 K 1_\nn)=0
$$
where $\hat{B}=\left(\hat{b}\left(z_i\right)\right)_{i=1}^{n_z}$, $1_\nn$ is the vector of size $\nn$ whose components are all equal to 1 and $\op_1\op_2 K=\left[\op_1\op_2 k(v_i,v_j)\right]_{i,j=1}^n$. This equation simplifies as
$$
\lambda\alpha+\frac{2}{n_z}\left( K\alpha+\beta n \hat{B}\right)=0\quad \text{and}\quad \beta =\frac{2}{\nn \lambda}
$$
from which the desired result follows.
\end{proof}

\begin{proof}[Proof of Equation \eqref{eq delta12}]
By the Cauchy-Schwarz inequality we have
$$
\left\|\nabla F(f)\right\|^2_{Z}
=\langle \LS \nabla F(f),\nabla F(f)\rangle_{\HH}\leq  \left\|\LS\nabla F(f)\right\|_{\HH}\left\|\nabla F(f)\right\|_{\HH}.
$$
Rearranging this inequality yields
\begin{equation*}
\delta^{(2)}_{}=\frac{\left\|\nabla F(f)\right\|^2_{Z}}{2\left\|\LS\nabla F(f)\right\|^2_{\HH}}
\leq\frac{\left\| \nabla F(f)\right\|^2_{\HH}}{2\left\|\nabla F(f)\right\|^2_{Z}} =\delta^{(1)}_{}. 
\qedhere
\end{equation*}
\end{proof}

\begin{proof}[Proof of Lemma \ref{lem 1}]
Recall that $b=\E\left[\op_2 k(V)\right]$ and $\hat{b}=1/n \sum_{i=1}^n \op_2 k(v_i)$. By Assumption (A4), we have $\left\|\op_2 k(v)\right\|_\HH\leq C$ for all $v\in \mathcal{V}$ and $\E\left[\left\|\op_2 k(V)\right\|_\HH^2\right]\leq C^2$.  All the assumptions of Proposition 2 in \cite{Caponnetto07} are thus verified with $L=2C$ and $\sigma=C$ and the desired result follows.
\end{proof}

\begin{proof}[Proof of Lemma \ref{lem 2}]
The space of Hilbert-Schmidt operators between $\HH$ and $\HH$ can be identified with $\HH\otimes \HH$, which is again a separable Hilbert space with norm $\left\|\cdot\right\|_{\text{HS}}$. Define the operator $\zeta_z:\HH\to\HH$ by $\zeta_z f=\left\langle f,k(\cdot,z)\right\rangle_{\HH}k(\cdot,z)=f(z)k(\cdot,z)$. Observe that $\LS=\E[\zeta_Z]$. To apply Proposition 2 from \cite{Caponnetto07}, we have to show that there exists $L>0$ such that $\left\|\zeta_z\right\|_{\text{HS}}\leq L/2$ for all $z\in\mathcal{Z}$. We first consider Setting (S1). Let $(e_i)_{i\in\N}$ be an orthonormal basis of $\HH$.
The Hilbert-Schmidt norm of $\zeta_z$ is 
\begin{align*}
\left\|\zeta_z\right\|_{\text{HS}}^2&=\sum_{i\in\N}\left\|\zeta_z e_i\right\|^2_{\HH}
= \sum_{i\in\N}\left\|k(\cdot,z)e_i(z)\right\|^2_{\HH}
= \sum_{i\in\N} e_i(z)^2 k(z,z).
\end{align*}
Using that $k\leq \kappa^2$ and Parseval's identity, we obtain
\begin{align*}
\left\|\zeta_z\right\|_{\text{HS}}^2\leq \kappa^2 \sum_{i\in\N} e_i(z)^2 
= \kappa^2 \sum_{i\in\N}  \left\langle e_i,k(\cdot,z)\right\rangle_{\HH}^2
= \kappa^2\left\|k(\cdot,z)\right\|_{\HH}^2 
= \kappa^4.
\end{align*}
Note that this implies $\E[\left\|\zeta_Z\right\|^2_{\text{HS}}]<\kappa^4$.
All the assumptions of Proposition 2 in \cite{Caponnetto07} are thus verified with $L=2\kappa^2$ and $\sigma=\kappa^2$ and the desired result follows since $\left\|\cdot\right\|_{\HH\to\HH}\leq \left\|\cdot\right\|_{\mathrm{HS}}$. 

Now we consider the proof under Setting (S2). Let $(X_i)_{i=1}^n$ be an i.i.d.\ sample of $X$ and $(U_j)_{j=1}^{n_u}$ be an i.i.d.\ sample of $U$. Let 
$$
B_1=\left\|\frac{1}{n_u}\sum_{j=1}^{n_u}\frac{1}{n}\sum_{i=1}^n\zeta_{(X_i,U_j)}-\E\left[\frac{1}{n}\sum_{i=1}^n\zeta_{(X_i,U)}|X_1,\ldots,X_n\right]\right\|_{\HH\to\HH}\quad
$$
and
$$
B_2=\left\|\frac{1}{n}\sum_{i=1}^n\E\left[\zeta_{(X_i,U)}|X_i\right]-\LS\right\|_{\HH\to\HH}.
$$
Observe that $\left\|{1}/{n}\sum_{i}\zeta_{(x_i,u)}\right\|\leq {1}/{n}\sum_{i}\left\|\zeta_{(x_i,u)}\right\|\leq \kappa^2$ for all $u\in \mathcal{U}$ and $x_1,\ldots,x_n\in \mathcal{X}$ by Jensen's inequality and the first part of the proof and similarly $\left\|\E\left[\zeta_{(X,U)}|X=x\right]\right\|\leq \kappa^2$ for all $x\in \mathcal{X}$. By Proposition 2 in \cite{Caponnetto07}, we have 
$\PP[B_1\geq \frac{6\kappa^2}{\sqrt{n_u}}\log\frac{4}{\eta}|X_1,\ldots,X_n]\leq \eta/2$ and $\PP[B_2\geq \frac{6\kappa^2}{\sqrt{n}}\log\frac{4}{\eta}]\leq \eta/2$. Finally, observe that
$$
\PP\left[B_1\geq \frac{6\kappa^2}{\sqrt{n_u}}\log\frac{4}{\eta}\right]
=\E\left[\PP\left[B_1\geq \frac{6\kappa^2}{\sqrt{n_u}}\log\frac{4}{\eta}\bigg|X_1,\ldots,X_n\right]\right]
\leq\E\left[\eta/2\right] \leq \eta/2.
$$
Now we compute
\begin{align*}
\PP\left[\left\|\frac{1}{nn_u}\sum_{i,j}\zeta_{(x_i,u_j)}-\LS\right\|
\geq \left(\frac{6\kappa^2}{\sqrt{n_u}}+\frac{6\kappa^2}{\sqrt{n}}\right)\log\frac{4}{\eta}\right]
&\leq \PP\left[B_1+B_2 \geq \left(\frac{6\kappa^2}{\sqrt{n_u}}+\frac{6\kappa^2}{\sqrt{n}}\right)\log\frac{4}{\eta}\right]\\
&\leq  \PP\left[\left\{B_1\geq  \frac{6\kappa^2}{\sqrt{n_u}}\log\frac{4}{\eta}\right\}
\cup \left\{B_2\geq  \frac{6\kappa^2}{\sqrt{n}}\log\frac{4}{\eta}\right\} \right]\\
&\leq \PP\left[B_1 \geq  \frac{6\kappa^2}{\sqrt{n_u}}\log\frac{4}{\eta}\right]
+\PP\left[B_2 \geq  \frac{6\kappa^2}{\sqrt{n}}\log\frac{4}{\eta}\right]\\
&\leq \eta. \qedhere
\end{align*}
\end{proof}

\begin{proof}[Proof of Theorem \ref{theo main}]
In this proof, we write $\left\|\cdot\right\|$  instead of $\left\|\cdot\right\|_{\HH}$. 
Let $m=n_z$ under Setting (S1) and $m=n_u$ under Setting (S2). We have
\begin{align*}
\left\|\hat{f}-f_\HH\right\|&\leq \left\|g_\lambda(\hat{\LS})(\hat{b}-b)\right\|+\left\|g_\lambda(\hat{\LS})\LS f_\HH-f_\HH\right\|\\
&\leq \left\|g_\lambda(\hat{\LS})(\hat{b}-b)\right\|+\left\|g_\lambda(\hat{\LS})(\LS-\hat{\LS})f_\HH\right\|+\left\|r_\lambda(\hat{\LS})f_\HH\right\|
\end{align*}
where $r_\lambda(\sigma)=g_\lambda(\sigma)\sigma-1$. Assume first that $r\leq 1$. On the one hand, as $\left\|g_\lambda(\hat{\LS})\right\|\leq B/\lambda$, we have using Lemmas \ref{lem 1} and \ref{lem 2}
\begin{align*}
\left\|g_\lambda(\hat{\LS})(\hat{b}-b)\right\|+\left\|g_\lambda(\hat{\LS})(\LS-\hat{\LS})f_\HH\right\|
&\leq \left\|g_\lambda(\hat{\LS})\right\| \left\|\hat{b}-b\right\|+\left\|g_\lambda(\hat{\LS})\right\| \left\|\LS-\hat{\LS}\right\| \left\|f_\HH\right\|\\
&\leq \left(\frac{A_1}{\lambda\sqrt{n}}+\frac{A_2}{\lambda\sqrt{m}}\right)\log \frac{4}{\eta}
\end{align*}
where $A_1$ and $A_2$ are constants that do not depend on $n$, $m$, $\eta$ and $\lambda$. 
On the other hand, since $\left\|r_\lambda(\hat{\LS})\right\|\leq \gamma$ and $\left\|r_\lambda(\hat{\LS})\hat{\LS}^r\right\|\leq \gamma_r\lambda^r$ by definition of a regulariser, we have
\begin{align*}
\left\|r_\lambda(\hat{\LS})f_\HH\right\|&= \left\|r_\lambda(\hat{\LS})\LS^r u\right\|\\
&\leq \left\|r_\lambda(\hat{\LS})\hat{\LS}^r u\right\|+\left\|r_\lambda(\hat{\LS})(\LS^r-\hat{\LS}^r) u\right\|\\
&\leq \gamma_r\lambda^r\left\|u\right\|+ \gamma \left\|\LS^r-\hat{\LS}^r\right\| \left\|u\right\|.
\end{align*}
By Theorem 1 in \cite{Bauer}, there exists $c_r>0$ such that
\begin{equation}
\label{eq Lr}
\left\|\LS^r-\hat{\LS}^r\right\|\leq c_r\left\|\LS-\hat{\LS}\right\|^r 
\leq c_r \left\{\left(\frac{6\kappa^2}{\sqrt{n}}+\frac{6\kappa^2}{\sqrt{m}}\right)\log\frac{4}{\eta}\right\}^r
\leq c_r \lambda^r
\end{equation}
where we have chosen $\lambda>6\kappa^2\left(1/{\sqrt{m}}+1/{\sqrt{n}}\right)\log({4}/{\eta})$ to obtain the last inequality. 
Gathering the above inequalities, we can conclude that 
$$
\left\|\hat{f}-f_\HH\right\|
\leq \left(\frac{A_1}{\lambda \sqrt{n}}+\frac{A_2}{\lambda \sqrt{m}}+A_3\lambda^r \right)\log\frac{4}{\eta} 
$$
where $A_3$ does not depend on $n$, $m$, $\eta$ and $\lambda$. 
Choosing $\lambda=(n^{-1/2}+m^{-1/2})^{1/(r+1)}$, we obtain the desired bound.  
To ensure that $\lambda>6\kappa^2\left(1/{\sqrt{m}}+1/{\sqrt{n}}\right)\log({4}/{\eta})$, we require that $(n^{-1/2}+m^{-1/2})^{-r/(r+1)}>6\kappa^2\log(4/\eta)$, which can be written as $\lambda^{-r}>6\kappa^2\log(4/\eta)$. 

When $r>1$, Theorem 1 in \cite{Bauer} cannot be applied but we have that $\left\|\LS^r-\hat{\LS}^r\right\|\leq c'_r\left\|\LS-\hat{\LS}\right\|$ by Lemma D.5 in \cite{Zhou} for some $c'_r>0$. Then, we can replace Equation \eqref{eq Lr} by
\begin{equation*}
\left\|\LS^r-\hat{\LS}^r\right\|\leq c'_r\left\|\LS-\hat{\LS}\right\| \leq c'_r \left(\frac{6\kappa^2}{\sqrt{n}}+\frac{6\kappa^2}{\sqrt{m}}\right)\log\frac{4}{\eta}
\end{equation*}
and a similar reasoning applies.
The proof of the convergence in $L^2(Z)$ follows similar arguments and is based on the identity $$\left\|\hat{f}-f_\HH\right\|_{L^2(Z)}= \left\|\sqrt{\LS}\left(\hat{f}-f_\HH\right)\right\|_{\HH}.$$ A complete proof is provided in \cite{Bauer}. In their notation, it suffices to replace $T$ by $\LS$, $T_\mathbf{x}$ by $\hat{\LS}$, $S_\mathbf{x}^*y$ by $\hat{b}$ and to set $\phi(t)=t^r$. 

\end{proof}



\end{document}